\documentclass[12pt]{article}
\usepackage{amsmath,graphicx}

\usepackage{hyperref,amsthm,amssymb}%,pstricks}
\theoremstyle{plain}
\newtheorem{theorem}{Theorem}
\newtheorem{corollary}[theorem]{Corollary}
\newtheorem{lemma}[theorem]{Lemma}
\newtheorem{proposition}[theorem]{Proposition}

\theoremstyle{definition}
\newtheorem{definition}[theorem]{Definition}

\newcommand{\seqnum}[1]{\href{http://oeis.org/#1}{\underline{#1}}}
\title{{\bf A critical quartet for queuing couples}}

\author{Donovan Young\\
\small St Albans, Hertfordshire, UK\\[-0.8ex] 
\small\tt donovan.m.young@gmail.com\\}

\begin{document}

\maketitle

\begin{abstract}
  We enumerate arrangements of $n$ couples, i.e. pairs of people,
  placed in a single-file queue, and consider four statistics from the
  vantage point of a distinguished given couple. In how many
  arrangements are exactly $p$ of the $n-1$ other couples i)
  interlaced with the given couple, ii) contained within them, iii)
  containing the given couple, and iv) lying outside the given couple?
  We provide generating functions which enumerate these arrangements
  and obtain the associated continuous asymptotic distributions in the
  $n\to\infty$ limit. The asymptotic distributions corresponding to
  cases i), iii), and iv) evince critical phenomena around the value
  $p_c=(n-1)/2$, such that the probability that 1) the couple is
  interlaced with more than half of the other couples, and 2) the
  couple is contained by more than half of the other couples, are both
  zero in the strict $n\to\infty$ limit. We further show that the
  cumulative probability that less than half of the other couples lie
  outside the given couple is $\pi/4$ in the limit, and that the
  associated distribution is uniform for $p<p_c$ .
\end{abstract}

%%%%%%%%%%%%%%%%%%%%%%%%%%%%%%%%%%%%%%%%%%%%%%%%%%%%%%%%%%%%%%%%%%%%%%%
\section{Introduction and main results}

The purpose of this paper is to study linear arrangements of $n$
distinguishable pairs of objects, treating the two members of a pair
as indistinguishable. The connection to linear chord diagrams is
immediate, as we can represent the pairs as chords joining two of $2n$
vertices laid out in a line, see Figure \ref{Fig:example}. The main
difference is that we treat the $n$ chords, ab initio, as
distinguishable.

The study of (indistinguishable) chord diagrams has a rich
history\footnote{The interested reader is directed to Pilaud and
  Ru\'{e} \cite{PR} for a more complete list of references.}. Touchard
\cite{T} and Riordan \cite{R} enumerated configurations by the total
number of crossings, and the limiting Normal distribution was obtained
by Flajolet and Noy \cite{FN}. More recently Pilaud and Ru\'{e}
\cite{PR} have extended the study of crossings in several
directions. Kreweras and Poupard \cite{KP} enumerated configurations
by the number of so-called short pairs, where adjacent vertices are
joined by a chord, finding that they are asymptotically Poisson in
distribution; c.f. Cameron and Killpatrick \cite{CK} and Krasko and
Omelchenko \cite{KO} for more modern treatments. 

We will enumerate configurations from the vantage point of a
distinguished {\it given} pair (which might appear in any position)
according to the relative position of the remaining $n-1$ pairs. Each
of these remaining pairs can be in one of four relative positions: i)
interlaced with the given pair, ii) entirely contained within the
endpoints of the given pair, iii) arching over the given pair and
hence entirely containing it, or iv) positioned entirely outside,
either to the left, or to the right, of the given pair.
\begin{figure}[t]
\begin{center}
\includegraphics[bb=0 0 282 73, height=1.25in]{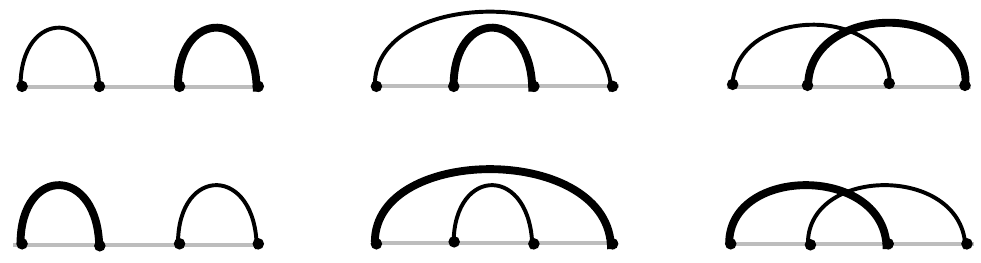}
\end{center}
\caption{The 6 configurations for the case $n=2$. The given pair is
  indicated as a bold arc. There are 4 configurations where the given
  pair is not crossed, hence $K_{2,0}=4$, whilst there are 2 where it
  is crossed once, hence $K_{2,1}=2$. Similarly $C_{2,0}=G_{2,0}=5$,
  $C_{2,1}=G_{2,1}=1$, and $X_{2,0}=4$, $X_{2,1}=2$.}
\label{Fig:example}
\end{figure}
It is clear that the total number of arrangements of the $n$
distinguishable pairs is $n!\,(2n-1)!!$, as there are $(2n-1)!!$
different linear chord diagrams. Due to the fact that we are
essentially interested in a single marked pair, i.e. the given pair,
we can safely paint the remaining $n-1$ pairs with the same brush and
treat them as indistinguishable -- this yields $n\,(2n-1)!!$
configurations, and is the number of linear chord diagrams with one
marked chord.
\begin{definition}
  A pair is said to be {\it crossed} by another pair if the other
  pair has one endpoint contained within the first pair, and the other
  outside of it, i.e. the two pairs are interlaced.
\end{definition}
\begin{definition}
  A pair is said to be {\it contained} by another pair if its
  endpoints are both located within the endpoints of the other pair.
\end{definition}
\begin{definition}
  A pair is said to be {\it containing} another pair if the other pair
  is contained by it.
\end{definition}
\begin{definition}
  A pair is said to be {\it excluded} by another pair if its endpoints
  are both to the left, or both to the right of the other pair.
\end{definition}
Amongst the $n\,(2n-1)!!$ arrangements, let there be $K_{n,p}$ where
exactly $p\in[0,n-1]$ of the remaining $n-1$ pairs are crossed by the
given pair. Similarly we define $C_{n,p}$, $G_{n,p}$, and $X_{n,p}$ to
be the number of configurations where exactly $p$ of the remaining
pairs are, respectively, contained by, containing, and finally
excluded by the given pair; see Figure \ref{Fig:example}.\\\\
\begin{figure}[h]
\begin{center}
  \includegraphics[bb=76 222 536 568, width=3.in]{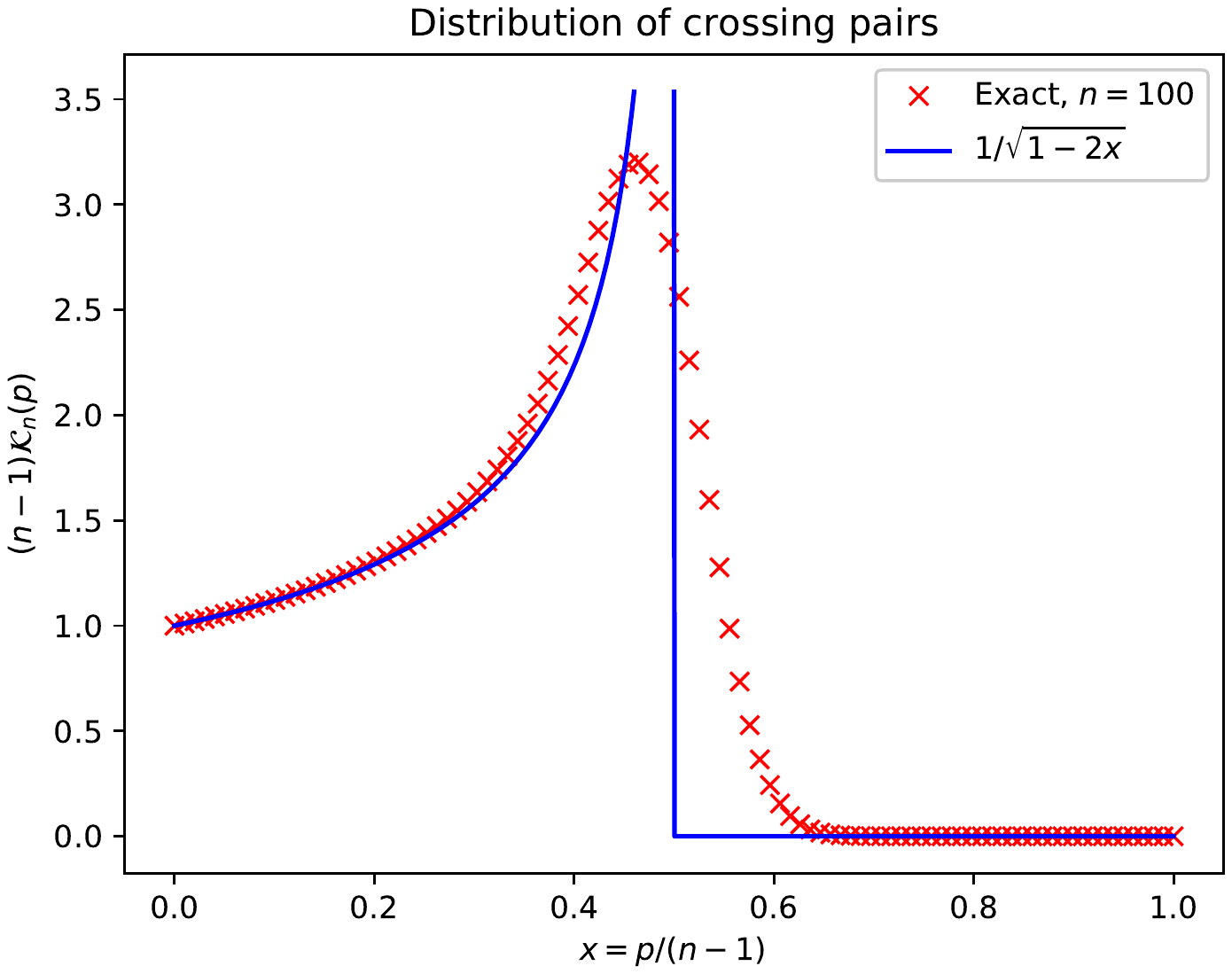}
  \includegraphics[bb=76 222 536 568, width=3.in]{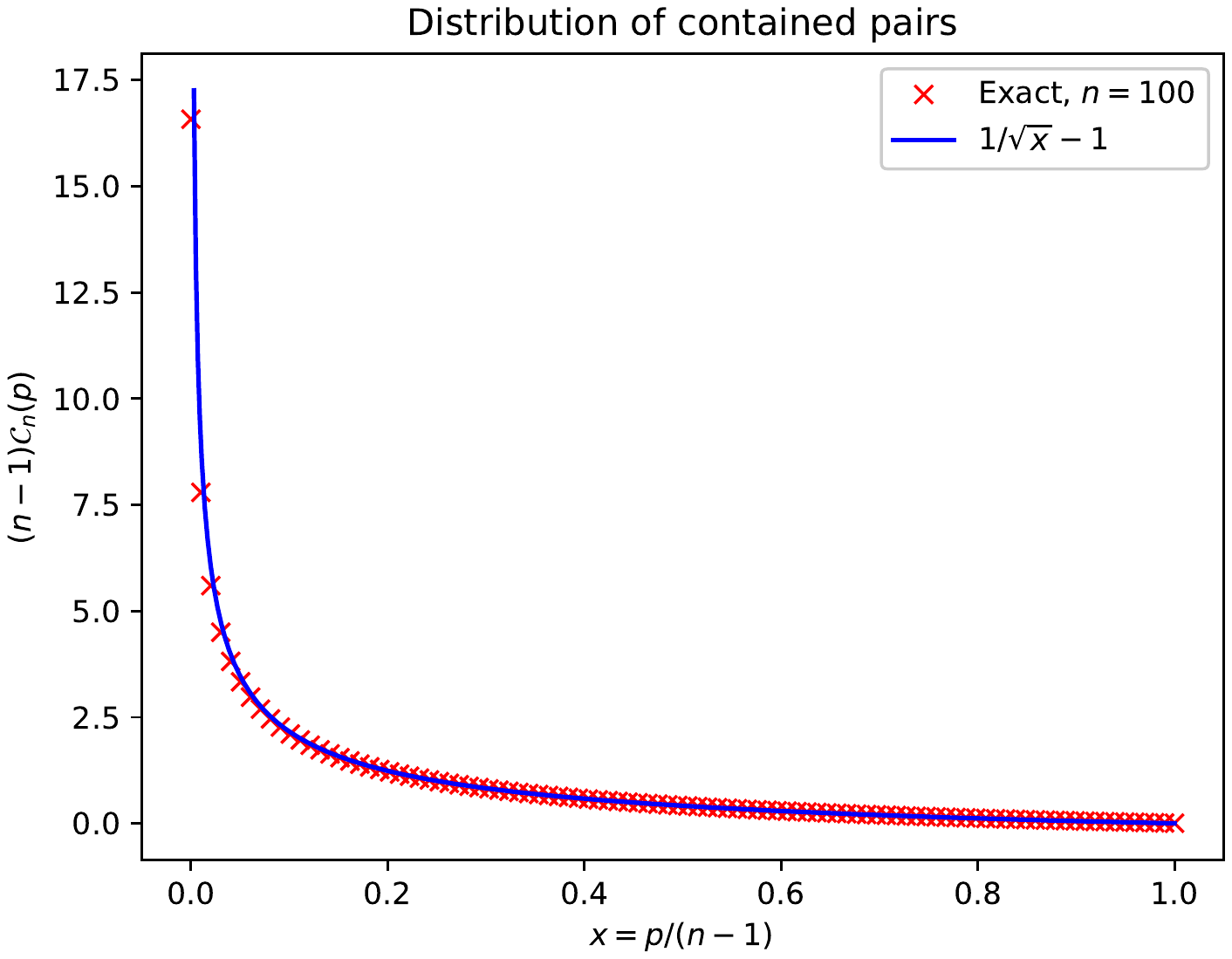}\\
  \includegraphics[bb=76 222 536 568, width=3.in]{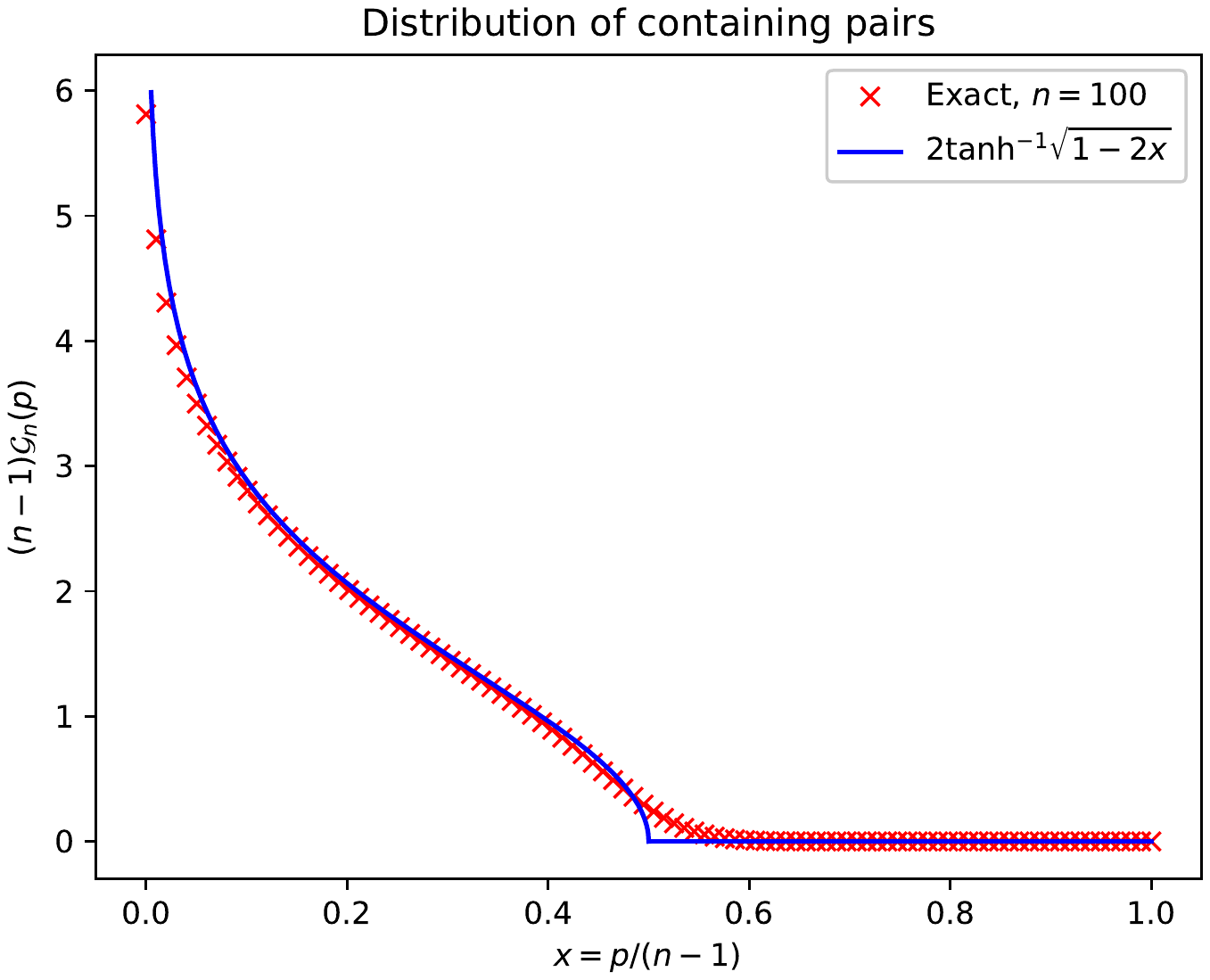}
  \includegraphics[bb=76 222 536 568, width=3.in]{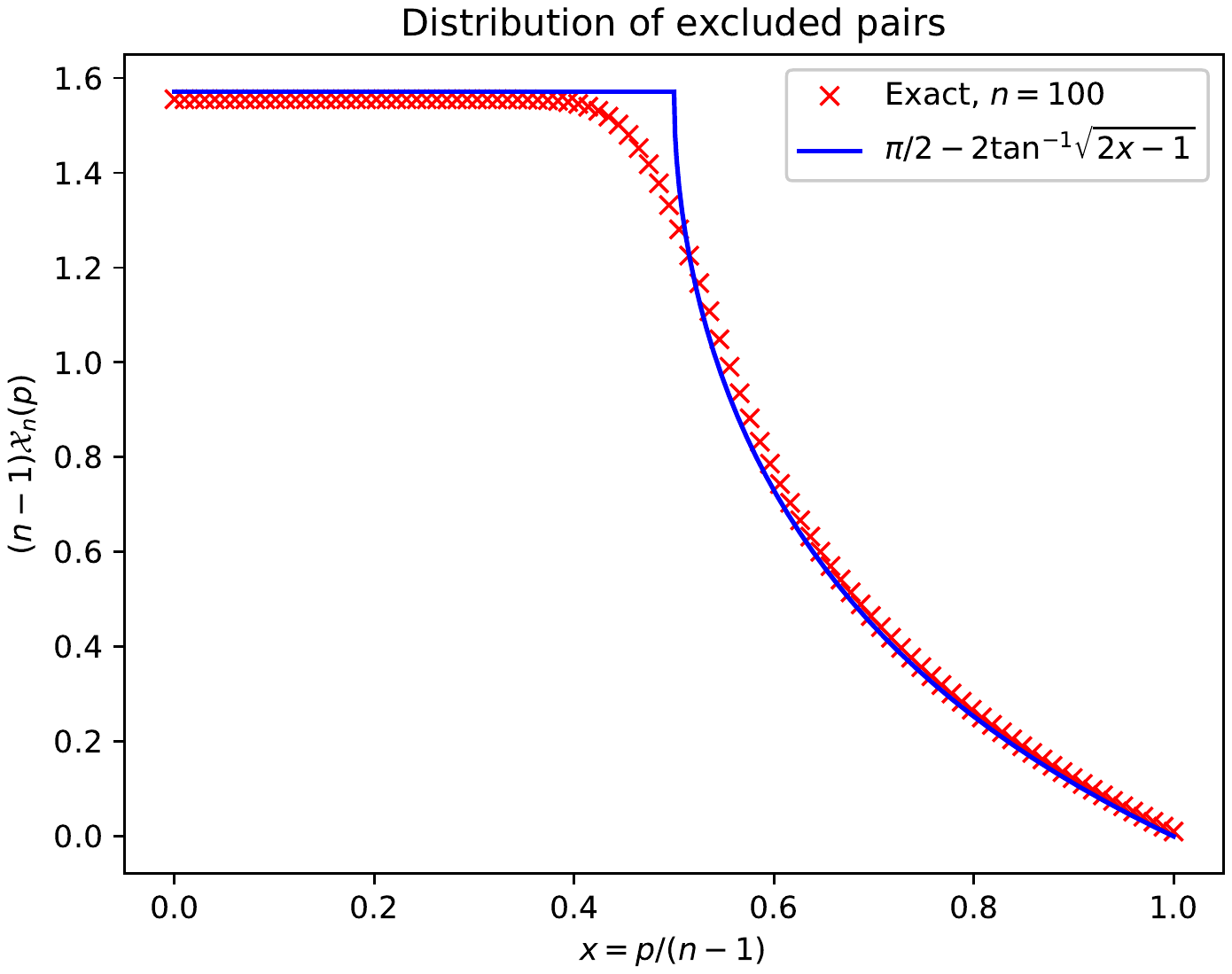}
\end{center}
\caption{The quartet of distributions. On the top row: on the left the
  distribution of pairs crossing the given pair, on the right pairs
  contained within the given pair. On the bottom row: on the left
  pairs which contain the given pair, on the right pairs situated
  outside the given pair. In each case the solid blue line is the
  asymptotic distribution, while the red ``x'' is the discrete value
  from the exact distribution for $n=100$.}
\label{Fig:dists}
\end{figure}
\noindent{\bf Generating functions} We define exponential generating
functions as follows
\begin{equation}\nonumber
   K(y,z)=  \sum_{n\geq 1}\sum_{p=0}^{n-1} K_{n,p}\, y^p \,\frac{z^n}{n!},
\end{equation}
and similarly for the $C_{n,p} \to C(y,z)$, $G_{n,p}\to G(y,z)$, and
$X_{n,p}\to X(y,z)$. In Theorems \ref{Thm:Kyz}, \ref{Thm:Cyz},
\ref{Thm:Gyz}, and \ref{Thm:Xyz} we prove that
\begin{equation}\nonumber
  \begin{split}
    &K(y,z) = \frac{z}{\sqrt{1-2z}\left(1-z(1+y)\right)},\quad
    C(y,z) = \frac{\sqrt{1-2yz}-\sqrt{1-2z}}{(1-2z)(1-y)},\\
    &G(y,z) = \frac{1}{(1-y)\sqrt{1-2z}}\ln\left(\frac{1-z(1+y)}{1-2z}\right),\\
    &X(y,z) = \frac{1}{(1-y)\sqrt{1-2z}}
  \tan^{-1}\frac{(1-y)z}{\sqrt{(1-2z)(1-2yz)}}.
  \end{split}
 \end{equation}
\noindent The form of $K(y,z)$ implies the recursion relation $K_{n,p}
= n \,K_{n-1,p} + n \,K_{n-1,p-1}$, $K_{n,0} =[z^n]K(0,z)$.\\\\
\noindent{\bf Asymptotic distributions}  We will also be interested in the associated
discrete probability distributions
\begin{equation}\nonumber
  P(\text{exactly $p$ pairs cross the given pair})={\cal K}_n(p)
  = \frac{1}{n\,(2n-1)!!}\,K_{n,p},
\end{equation}
and so for ${\cal C}_n(p)$, ${\cal G}_n(p)$, and ${\cal X}_n(p)$,
where we treat all $n\,(2n-1)!!$ arrangements as equally likely. In
the limit as $n\to\infty$ we define a continuous real variable
$x=\lim_{n\to\infty}p/(n-1) \in [0,1]$, and an associated continuous
probability distribution
\begin{equation}\nonumber
  {\cal K}(x) = \lim_{n\to\infty} (n-1)\,{\cal K}_n\left((n-1)x\right),
\end{equation}
and so for ${\cal C}(x)$, ${\cal G}(x)$, and ${\cal X}(x)$. In
Theorems \ref{Thm:Kx}, \ref{Thm:Cx}, \ref{Thm:Gx}, and \ref{Thm:Xx} we
prove that
\begin{equation}\nonumber
  \begin{split}
 & {\cal K}(x) = \begin{cases}
   1/\sqrt{1-2x}& 0\leq x < 1/2 \\
    0& 1/2 \leq x \leq 1
    \end{cases},\qquad
  {\cal C}(x) = \frac{1}{\sqrt{x}}-1, \quad 0 < x \leq 1,\\
&  {\cal G}(x)=
  \begin{cases}
    2\tanh^{-1}\sqrt{1-2x}& 0< x \leq 1/2 \\
    0& 1/2 < x \leq 1
    \end{cases},\\
&  {\cal X}(x)=
  \begin{cases}
    \pi/2& 0\leq x < 1/2 \\
    \pi/2-2\tan^{-1}\sqrt{2x-1}& 1/2 \leq x \leq 1
  \end{cases}.
  \end{split}
\end{equation}

In Figure \ref{Fig:dists} the four distributions are shown. It is
remarkable that ${\cal K}(x)$, ${\cal G}(x)$, and ${\cal X}(x)$ all
show critical phenomena\footnote{For an introduction to critical
  phenomena, see \cite{G}. The term is usually reserved for the
  observation of a sharp transition in a system when a control
  variable is adjusted beyond a critical value; we are using it in a
  slightly more general manner here.} at $x=1/2$, corresponding to
half of the $n-1$ pairs.  This is most striking in the discontinuity
observed in ${\cal K}(x)$, where the asymptotic probability that the
given pair is crossed by more than half of the remaining pairs is
zero, while the mode of the distribution is also half of the remaining
pairs. In ${\cal G}(x)$ we see that the asymptotic probability that
the given pair is contained within more than half of the remaining
pairs is also zero. The distribution ${\cal X}(x)$ shows that the
asymptotic (cumulative) probability that less than half of the
remaining pairs are outside the given pair is given by $\pi/4$, while
the distribution itself is uniform in this region.

In Lemmas \ref{Lem:Kfact}, \ref{Lem:Cfact}, \ref{Lem:Gfact}, and
\ref{Lem:Xfact}, we obtain expressions for the $m^\text{th}$ factorial
moments of the exact distributions. In particular,
\begin{equation}\nonumber
  \begin{split}
&  \sum_{p=0}^{n-1} \frac{p!}{(p-m)!}\,{\cal K}_n(p)
    = \frac{(n-1)!}{(n-m-1)!}\frac{m!}{(2m+1)!!},\\
&  \sum_{p=0}^{n-1} \frac{p!}{(p-m)!}\,{\cal C}_n(p)
  =\frac{(n-1)!}{(n-m-1)!}\frac{1}{(m+1)(2m+1) },\\
&\sum_{p=0}^{n-1} \frac{p!}{(p-m)!}\,{\cal G}_n(p)=
  \frac{(n-1)!}{(n-m-1)!}\frac{m!}{(m+1)(2m+1)!! },\\
&\sum_{p=0}^{n-1} \frac{p!}{(p-m)!}\,{\cal X}_n(p)=
  \frac{(n-1)!}{(n-m-1)!}\frac{1}{m+1}\int_{1/2}^1dx\,\frac{x^m}{\sqrt{2x-1}}.
  \end{split}
\end{equation}
The mean values for the four distributions tell us that, on average, a
third of the remaining pairs cross the given pair, a sixth are
contained by it, another sixth contain the given pair, and the
remaining third are excluded by it.

%%%%%%%%%%%%%%%%%%%%%%%%%%%%%%%%%%%%%%%%%%%%%%%%%%%%%%%%%%%%%%%%%%%%%%%%%
\section{Enumeration by crossings}

\begin{definition}
  We define the {\it size} of a pair to be the number of vertices
  contained between its endpoints; the minimum size is zero, while the
  maximum size achievable is $2n-2$.
\end{definition}

\noindent{\bf Distribution of sizes} There are clearly $(2n - d - 1)$ positions
a given pair of size $d$ can occupy. Once placed, there are $(2n-3)!!$
ways of placing the remaining $(n-1)$ indistinguishable pairs. The
probability ${\cal S}_n(d)$ that the given pair has size $d$ is therefore
\begin{equation}\nonumber
  {\cal S}_n(d) = (2n-d-1)\frac{(2n-3)!!}{n\,(2n-1)!!}
  = \frac{1}{n}\left(1 - \frac{d}{2n-1}\right),
\end{equation}
which is a trapezoidal distribution. Straightforward computations
yield a mean of $2(n-1)/3$, or a third of the maximal distance, and a
variance of $(2n+1)(n-1)/9$.\\

\noindent{\bf Counting by crossings} The minimum number of times a
given pair can be crossed is zero -- this is when all its contained
vertices are matched amongst one another, and so with all its excluded
vertices. The maximum number of times a given pair can be crossed is
$n-1$ as there are $2n-2$ vertices other than those occupied by the
endpoints of the given pair, and to achieve the maximal crossing we
require half of them to be contained (i.e. the given pair has a size
of $n-1$) and then each matched with one of the $n-1$ excluded
vertices. Let $p$ be the number of times a given pair of size $d$ is
crossed. It is clear that $p\equiv d \pmod{2}$.
\begin{proposition}\label{Prop:Knpd}
The number $K_{n,p,d}$ of configurations in which the given pair has size $d$, and
is crossed by $p$ other pairs, is given by
\begin{equation}\nonumber
K_{n,p,d}= \frac{2^{p-n+1}\,d!\,(2n-d-1)!}
  {p!\,\left(n-1-\frac{d-p}{2}\right)!\,\left(\frac{d-p}{2}\right)!},
\end{equation}
where $d\,\mathrm{mod}\, 2 \leq p \leq \min(d,2n-d-2)$, and $0 \leq d \leq 2n-2$.
\end{proposition}
\begin{proof}
In order to enumerate configurations where a given pair of size
$d$ is crossed $p$ times, we consider the $d$ contained vertices, and
choose $p$ of these to be matched with another selection of $p$
excluded vertices. The remaining contained vertices are then matched
amongst themselves, and so for the remaining excluded vertices.
\begin{itemize}
\item There are $p!{d\choose p}{2n-d-2\choose p}$ ways of choosing the
  $p$ contained and $p$ excluded vertices and then matching them up.
\item There are $(d-p-1)!!\,(2n-d-p-3)!!$ ways of matching the remaining vertices.
\item There are $(2n - d - 1)$ positions for the given pair to occupy.
\end{itemize}
We therefore have that
\begin{equation}\nonumber
K_{n,d,p}= 
  {d\choose p}{2n-d-2\choose p}\,p!\,(d-p-1)!!\,(2n-d-p-3)!!\,
    (2n - d - 1).
\end{equation}
Using the identity $(2n-1)!! = (2n)!/(n!2^n)$, and simplifying this
expression, we obtain the desired result.
\end{proof}
\begin{table}
\begin{center}
  \begin{tabular}{c|llllll}
  $n$ \textbackslash$p$& 0& 1& 2& 3& 4& 5\\
  \hline  
  1& 1\\
  2& 4& 2\\
  3& 21& 18& 6\\
  4& 144& 156& 96& 24\\
  5& 1245& 1500& 1260& 600& 120\\
  6& 13140& 16470& 16560& 11160& 4320& 720\\
\end{tabular}
\end{center}
\caption{The numbers $K_{n,p}$, \seqnum{A336598} in the OEIS, to
  appear. The first column is \seqnum{A233481}. The leading diagonal
  are the factorials $n!$.}
\end{table}
\begin{lemma}\label{Lem:Knp}
The number $K_{n,p}$ of configurations in which the given pair is
crossed by $p$ other pairs, is given by
\begin{equation}\nonumber
  K_{n,p}=n\,(2n-1)!!
 \int_0^1 d\alpha \, 2(1-\alpha) \,{n-1\choose p}
\left(2\alpha(1-\alpha)\right)^p \left( 1 -  2\alpha(1-\alpha)\right)^{n-p-1}.
\end{equation}
\end{lemma}
\begin{proof}
We sum the result of Proposition \ref{Prop:Knpd} over sizes $d$ to
produce $K_{n,p}$. For fixed $p$, we must sum $d$ over the range $p
\leq d \leq 2n - 2 - p$, where $d$ is incremented by $2$ in each
successive term. To make this summation more convenient we write $d =
2k + p$ and sum $k$ over $0 \leq k \leq n-p-1$:
\begin{equation}\nonumber
K_{n,p} =  \frac{2^{p-n+1}}{p!}\sum_{k=0}^{n-p-1}
  \frac{(2k + p)!\, (2n - 2k - p - 1)!}{(n-p-k-1)!\,k!}.
\end{equation}
We now exploit the following integral representation of the Euler Beta function:
\begin{equation}\nonumber
  \frac{(2k + p)!\, (2n - 2k - p - 1)!}{(2n)!} = \int_0^1 d\alpha \,
  \alpha^{2k+p} \, (1 - \alpha)^{2n - 2k - p - 1},
\end{equation}
to obtain
\begin{equation}\nonumber
  \begin{split}
K_{n,p} &=  \frac{2^{p-n+1}\,(2n)!}{p!}\int_0^1 d\alpha \,\sum_{k=0}^{n-p-1}
\frac{\alpha^{2k+p} \, (1 - \alpha)^{2n - 2k - p - 1}}{(n-k-p-1)!\,k!}\\
&= \frac{2^{p-n+1}\,(2n)!}{p!}\int_0^1 d\alpha \,
\sum_{k=0}^{n-p-1}
\frac{\alpha^p \, (1 - \alpha)^{2n - p - 1}}{(n-k-p-1)!\,k!}\,
\left(\frac{\alpha^2}{(1-\alpha)^2}\right)^k\\
&= \frac{2^{p-n+1}\,(2n)!}{p!}\int_0^1 d\alpha \,
\frac{\alpha^p \, (1 - \alpha)^{2n - p - 1}}{(n-p-1)!}\,
\left(1 + \frac{\alpha^2}{(1-\alpha)^2}\right)^{n-p-1}\\
&=n\,(2n-1)!!\, {n-1\choose p}\int_0^1 d\alpha \, 2(1-\alpha) \,
\left(2\alpha(1-\alpha)\right)^p \left( (1-\alpha)^2 + \alpha^2 \right)^{n-p-1}\\
&=n\,(2n-1)!! \int_0^1 d\alpha \, 2(1-\alpha) \,{n-1\choose p}
\left(2\alpha(1-\alpha)\right)^p \left( 1 -  2\alpha(1-\alpha)\right)^{n-p-1}.
\end{split}
\end{equation}
\end{proof}

\begin{theorem}\label{Thm:Kyz}
  The exponential generating function $K(y,z)$ is given by
  \begin{equation}\nonumber
    K(y,z)=  \sum_{n\geq 1}\sum_{p=0}^{n-1} K_{n,p}\, y^p \,\frac{z^n}{n!}
    = \frac{z}{\sqrt{1-2z}\left(1-z(1+y)\right)}.
    \end{equation}
\end{theorem}
\begin{proof}
  We sum the result of Lemma \ref{Lem:Knp} against $y^p$ to obtain
  \begin{equation}\nonumber
  \sum_{p=0}^{n-1} K_{n,p}\, y^p =
  n\,(2n-1)!!\int_0^1 d\alpha \, 2(1-\alpha) \, \bigl(
  1 - (1-y) 2\alpha(1-\alpha) \bigr)^{n-1}.
  \end{equation}
  We then perform the sum over $n$ against $z^n/n!$
  \begin{equation}\nonumber
  \begin{split}
 & \sum_{n,p} K_{n,p}\, y^p \,\frac{z^n}{n!} = \sum_n \frac{n\,(2n-1)!!}{n!}\,
z^n  \int_0^1 d\alpha \, 2(1-\alpha) \, \bigl(
1 - (1-y) 2\alpha(1-\alpha) \bigr)^{n-1}\\
&=\int_0^1 d\alpha \, 2(1-\alpha) \frac{z}
{\left(1 - 2 z \left(1 - (1-y) 2\alpha(1-\alpha)\right)\right)^{3/2}}
= \frac{z}{\sqrt{1-2z}\left(1-z(1+y)\right)}.
\end{split}
\end{equation}
\end{proof}

\begin{corollary}
  The $K_{n,p}$ obey the following recursion relation
\begin{equation}\nonumber
K_{n,p} = n \,K_{n-1,p} + n \,K_{n-1,p-1}, \qquad K_{n,0}
=[z^n]\frac{z}{\sqrt{1-2z}\left(1-z\right)},
\end{equation}
where we note that $K_{n,0}$ is $\seqnum{A233481}$ in the OEIS -- the
number of singletons (strong fixed points) in pair-partitions.
\end{corollary}
\begin{proof}
  The recursion relation is implied by the factor $1-z(1+y)$ in the
  denominator of the generating function $K(y,z)$.
\end{proof}

%It is also (nearly) a Riordan array, in the sense that the ordinary
%generating function for the numbers $N_{n+1,k}/(n+1)!$, i.e. $\sum
%z^n K_{n+1,k}/(n+1)!$ is $(zh(z))^k d(z)$ where $d(z) =
%1/(1-z)/\sqrt{1-2z}$ and $h(z) = 1/(1-z)$.

%---------------------------------------------------------------------------%
\subsection*{Probability distribution and asymptotics}

We define a discrete random variable $K$ which corresponds to the
number of pairs which cross the given pair. The result of Lemma
\ref{Lem:Knp} implies that the probability that $K$ takes the value
$p$ is given by
\begin{equation}\nonumber
{\cal K}_n(p) = \frac{K_{n,p}}{n\,(2n-1)!!}=
\int_0^1 d\alpha \, 2(1-\alpha) \,{n-1\choose p}
\left(2\alpha(1-\alpha)\right)^p \left( 1 -  2\alpha(1-\alpha)\right)^{n-p-1},
\end{equation}
which is an integral over Binomial distributions. In order to compute
the factorial moments of this distribution, we define a generating
function as follows
\begin{equation}\nonumber
  {\cal P}_n(y)=\sum_{p=0}^{n-1} {\cal K}_n(p)\, y^p =
  \int_0^1 d\alpha \, 2(1-\alpha) \, \bigl(
  1 - (1-y) 2\alpha(1-\alpha) \bigr)^{n-1}.
  %={}_2 F_1\left(1,1-n;\frac{3}{2};\frac{1-y}{2}\right).
\end{equation}
\begin{lemma}\label{Lem:Kfact}
The $m^\text{th}$ factorial moment of ${\cal K}_n(p)$ is given by
\begin{equation}\nonumber
  \sum_{p=0}^{n-1} \frac{p!}{(p-m)!}\,{\cal K}_n(p)
  = \frac{(n-1)!}{(n-m-1)!}\frac{m!}{(2m+1)!!}.
\end{equation}
In particular this provides the mean $E(K)=(n-1)/3$, and the
variance $\mathrm{Var}(K)=(n-1)(n+8)/45$. 
\end{lemma}
\begin{proof}
\begin{equation}\nonumber
\begin{split}
  &\sum_{p=0}^{n-1} \frac{p!}{(p-m)!}\,{\cal K}_n(p)= \left.\frac{d^m}{dy^m}\right|_{y=1} {\cal P}_n(y)
  = \int_0^1 d\alpha \, 2(1-\alpha) \, \frac{(n-1)!}{(n-m-1)!}
  \left( 2\alpha(1-\alpha)\right)^m\\
  &=2^{m+1}\frac{(n-1)!}{(n-m-1)!}\frac{m!\,(m+1)!}{(2m+2)!}
  = \frac{(n-1)!}{(n-m-1)!}\frac{m!}{(2m+1)!!}.
\end{split}
\end{equation}
\end{proof}

In the limit as $n\to\infty$ we define a continuous real variable
$x=\lim_{n\to\infty}p/(n-1) \in [0,1]$, and an associated continuous
probability distribution
\begin{equation}\nonumber
  {\cal K}(x) = \lim_{n\to\infty} (n-1)\,{\cal K}_n\left((n-1)x\right),
\end{equation}

\begin{theorem}\label{Thm:Kx}
  The asymptotic distribution ${\cal K}(x)$ is given by
\begin{equation}\nonumber
  {\cal K}(x) = \begin{cases}
   1/\sqrt{1-2x}& 0\leq x < 1/2 \\
    0& 1/2 \leq x \leq 1
    \end{cases}.
\end{equation}
\end{theorem}
\begin{proof}
The most satisfying proof of this fact is to show that the large-$n$
limit of the factorial moments is correctly reproduced. To wit,
\begin{equation}\nonumber
  \int_0^{1/2}dx\,\frac{x^m}{\sqrt{1-2x}}
  = \frac{1}{2^{m+1}}\int_0^1 du\, u^{-1/2} (1-u)^m
  =\frac{m!}{(2m+1)!!},
\end{equation}
where we have used the substitution $u = 1 - 2x$. Comparing to Lemma 
(\ref{Lem:Kfact}), we see that in the large-$n$ limit
$\frac{(n-1)!}{(n-m-1)!} \to n^m$, and so we are indeed recovering the
factorial moments correctly.
\end{proof}

\noindent{\bf Alternative proof} Another perspective is to return to the following
representation of the exact distribution
\begin{equation}\nonumber
 \int_0^1 d\alpha \, 2(1-\alpha) \,{n-1\choose p}
\left(2\alpha(1-\alpha)\right)^p \left( 1 -  2\alpha(1-\alpha)\right)^{n-p-1},
\end{equation}
and to use the Normal approximation of the Binomial distribution. When
$\alpha$ is near 0 or 1, this will not be a good approximation, but
this seems to be a set of small enough measure not to impact the
overall approximation for $n\to\infty$. We begin by changing the
integration variable $\alpha = \sin^2{\theta\over 2}$
\begin{equation}\nonumber
 \int_0^{\pi} d\theta \, \sin\theta\, \frac{1+\cos\theta}{2} \,{n-1\choose p}
\left(\frac{\sin^2\theta}{2}\right)^p \left(\frac{1+\cos^2\theta}{2}\right)^{n-p-1} ,
\end{equation}
where we note that $1 + \cos\theta$ may be replaced by $1$ as the rest
of the integrand is even about $\theta=\pi/2$. We now take an integral
over Normal distributions with mean ${1\over 2}(n-1)\sin^2\theta$ and
variance ${1\over 4}(n-1)\sin^2\theta\left(1+\cos^2\theta\right) =
{1\over 4}(n-1)\left(1-\cos^4\theta\right)$
\begin{equation}\nonumber
N(x)=
  \frac{\sqrt{n-1}}{\sqrt{2\pi}} \int_0^\pi \frac{d\theta}{\sqrt{1+\cos^2\theta}}
  \,\mathrm{Exp}\left(\frac{-2\,(n-1)
    \left(x-{1\over 2}\sin^2\theta\right)^2}{1-\cos^4\theta}\right).
\end{equation}
This distribution interpolates between the discrete values of the
actual distribution remarkably well, and the integral over $\theta$
converges well enough to allow for efficient numerical integration for
all values of $x$. It has a tail for $x<0$ which is suppressed for
large $n$. It is straightforward to show that all the moments match
the actual distribution in the strict $n\to\infty$ limit; $N(x)$ also has
the exact mean and variance, and the third moment is correct at ${\cal
  O}(n^{-1})$. Taking the $n\to\infty$ limit, we may use the method of
steepest descent to evaluate the integral. For $x\in[0,1/2)$, there
  are two saddle points located at the following values of $\theta$
\begin{equation}\nonumber
  \theta_0 = \arcsin \sqrt{2x}, \qquad
  \theta_1 = \pi - \arcsin \sqrt{2x},
\end{equation}
which yield the dominant contributions to the
integral\footnote{$\theta=\pi/2$ is also a saddle point, but the
  resulting contribution to the integral is exponentially
  suppressed for $x<1/2$.}. Representing $N(x)$ as
\begin{equation}\nonumber
\int d\theta\,f(\theta) \,e^{(n-1) S(\theta)},
\end{equation}
one finds that
\begin{equation}\nonumber
  \left. \frac{d^2S}{d\theta^2} \right|_{\theta = \theta_0}=
  \left. \frac{d^2S}{d\theta^2} \right|_{\theta = \theta_1}
  = -\frac{4\cos^2\theta_0}{1+\cos^2\theta_0},
\end{equation}
and so the two saddle points contribute the same result,
namely\footnote{Note that $S(\theta_0)=S(\theta_1)=0$.}
\begin{equation}\nonumber
  \begin{split}
 & \frac{\sqrt{2\pi}}{\sqrt{n-1}}\, f(\theta_0)\,
    \left(-\left. \frac{d^2S}{d\theta^2} \right|_{\theta =
      \theta_0}\right)^{-1/2}
    = \frac{\sqrt{2\pi}}{\sqrt{n-1}}\,
    f(\theta_1)\,\left(- \left. \frac{d^2S}{d\theta^2} \right|_{\theta =
      \theta_1}\right)^{-1/2}\\
    &= \frac{1}{2|\cos\theta_0|}
    =
    \frac{1}{2|\cos\theta_1|}=\frac{1}{2}\frac{1}{\sqrt{1-2x}},
  \end{split}
\end{equation}
and so the sum of the two contributions yields the desired result.

%%%%%%%%%%%%%%%%%%%%%%%%%%%%%%%%%%%%%%%%%%%%%%%%%%%%%%%%%%%%%%%%%%%%%%%%%
\section{Enumeration by contained pairs}

We now enumerate configurations according to the number $p$ of
pairs contained within the given pair. We begin by summing
the result of Proposition \ref{Prop:Knpd} over all possible crossings,
noting that if a contained vertex is not part of a crossing pair, it
is necessarily part of a contained pair. We let $d=2p+k$, so that the
number of crossings $k$ is bounded between $0 \leq k \leq n-p-1$.

\begin{lemma}\label{Lem:Cnp}
  The number $C_{n,p}$ of configurations in which the given pair
  contains $p$ other pairs, is given by
\begin{equation}\nonumber
  C_{n,p}=\sum_{k=0}^{n-p-1} K_{n,k,2p+k} = n\,(2n-1)!!
  \int_0^1 d\alpha \,2(1-\alpha) {n-1 \choose p} \left(\alpha^2\right)^p
  \left(1-\alpha^2\right)^{n-1-p}.
\end{equation}
\end{lemma}
\begin{proof}
We exploit the Euler Beta integral used in the proof of Lemma \ref{Lem:Knp}.
\begin{equation}\nonumber
  \begin{split}
C_{n,p} &=  \frac{2^{k-n+1}\,(2n)!}{p!}\int_0^1 d\alpha \,\sum_{k=0}^{n-p-1}
\frac{\alpha^{2p+k} \, (1 - \alpha)^{2n - 2p - k - 1}}{(n-k-p-1)!\,k!}\\
&= \frac{2^{-n+1}\,(2n)!}{p!}\int_0^1 d\alpha \,
\sum_{k=0}^{n-p-1}
\frac{\alpha^{2p} \, (1 - \alpha)^{2n - 2p - 1}}{(n-k-p-1)!\,k!}\,
\left(\frac{2\alpha}{1-\alpha}\right)^k\\
&=n\,(2n-1)!!\, {n-1\choose p}\int_0^1 d\alpha \, 2(1-\alpha) \,
\alpha^{2p} (1-\alpha)^{n-p}
\left( 1+\alpha\right)^{n-p-1}\\
&= n\,(2n-1)!!
  \int_0^1 d\alpha \,2(1-\alpha) {n-1 \choose p} \left(\alpha^2\right)^p
  \left(1-\alpha^2\right)^{n-1-p}.
\end{split}
\end{equation}
  \end{proof}

\begin{theorem}\label{Thm:Cyz}
  The exponential generating function $C(y,z)$ is given by
  \begin{equation}\nonumber
    C(y,z)=  \sum_{n\geq 1}\sum_{p=0}^{n-1} C_{n,p}\, y^p \,\frac{z^n}{n!}
= \frac{\sqrt{1-2yz}-\sqrt{1-2z}}{(1-2z)(1-y)}.
    \end{equation}
\end{theorem}
\begin{proof}
  We sum the result of Lemma \ref{Lem:Cnp} against $y^p$ to obtain
  \begin{equation}\nonumber
  \sum_{p=0}^{n-1} C_{n,p}\, y^p =
  n\,(2n-1)!!\int_0^1 d\alpha \, 2(1-\alpha) \, \bigl(
  1 - (1-y) \alpha^2\bigr)^{n-1}.
  \end{equation}
  We then perform the sum over $n$ against $z^n/n!$
  \begin{equation}\nonumber
  \begin{split}
 & \sum_{n,p} C_{n,p}\, y^p \,\frac{z^n}{n!} = \sum_n \frac{n\,(2n-1)!!}{n!}\,
z^n  \int_0^1 d\alpha \, 2(1-\alpha) \, \bigl(
1 - (1-y) \alpha^2 \bigr)^{n-1}\\
&=\int_0^1 d\alpha \, 2(1-\alpha) \frac{z}
{\left(1 - 2 z \left(1 - (1-y) \alpha^2\right)\right)^{3/2}}
= \frac{\sqrt{1-2yz}-\sqrt{1-2z}}{(1-2z)(1-y)}.
\end{split}
\end{equation}
\end{proof}
\begin{table}[t]
\begin{center}
  \begin{tabular}{c|llllll}
  $n$ \textbackslash$p$& 0& 1& 2& 3& 4& 5\\
  \hline  
  1& 1\\
  2& 5& 1\\
  3& 33& 9& 3\\
  4& 279& 87& 39& 15\\
  5& 2895& 975& 495& 255& 105\\
  6& 35685& 12645& 6885& 4005& 2205& 945\\
\end{tabular}
\end{center}
\caption{The numbers $C_{n,p}$, \seqnum{A336599} in the OEIS, to appear. The
  leading diagonal are the double factorials $(2n-3)!!$. The first column is
  \seqnum{A129890}. The second column is \seqnum{A035101}.}
\end{table}
%
%---------------------------------------------------------------------------%
\subsection*{Probability distribution and asymptotics}

We define a discrete random variable $C$ which corresponds to the
number of pairs which are contained by the given pair. The result of Lemma
\ref{Lem:Cnp} implies that the probability that $C$ takes the value
$p$ is given by
\begin{equation}\nonumber
{\cal C}_n(p) = \frac{C_{n,p}}{n\,(2n-1)!!}=
\int_0^1 d\alpha \, 2(1-\alpha) \,{n-1\choose p}
\left(\alpha^2\right)^p \left( 1 -  \alpha^2\right)^{n-p-1},
\end{equation}
which is an integral over Binomial distributions. In order to compute
the factorial moments of this distribution, we define a generating
function as follows
\begin{equation}\nonumber
  {\cal P}_n(y)=\sum_{p=0}^{n-1} {\cal C}_n(p)\, y^p =
\int_0^1 d\alpha \, 2(1-\alpha) \, \bigl(
  1 - (1-y) \alpha^2\bigr)^{n-1}.
  %\int_0^1 d\alpha \, 2(1-\alpha) \,{n-1\choose p}
%\left(\alpha^2\right)^p \left( 1 -  \alpha^2\right)^{n-p-1}.
\end{equation}
\begin{lemma}\label{Lem:Cfact}
The $m^\text{th}$ factorial moment of ${\cal C}_n(p)$ is given by
\begin{equation}\nonumber
  \sum_{p=0}^{n-1} \frac{p!}{(p-m)!}\,{\cal C}_n(p)
=\frac{(n-1)!}{(n-m-1)!}\frac{1}{(m+1)(2m+1) }.
\end{equation}
In particular this provides the mean $E(C)=(n-1)/6$, and the variance
$\mathrm{Var}(C)=(n-1)(7n+11)/180$.
\end{lemma}
\begin{proof}
\begin{equation}\nonumber
\begin{split}
  &\sum_{p=0}^{n-1} \frac{p!}{(p-m)!}\,{\cal C}_n(p)
  = \left.\frac{d^m}{dy^m}\right|_{y=1} {\cal P}_n(y)
  = \int_0^1 d\alpha \, 2(1-\alpha) \, \frac{(n-1)!}{(n-m-1)!}
  \,\alpha^{2m}\\
  &=\frac{(n-1)!}{(n-m-1)!}\frac{1}{(m+1)(2m+1) }.
\end{split}
\end{equation}
\end{proof}

In the limit as $n\to\infty$ we define a continuous real variable
$x=\lim_{n\to\infty}p/(n-1) \in [0,1]$, and an associated continuous
probability distribution
\begin{equation}\nonumber
  {\cal C}(x) = \lim_{n\to\infty} (n-1)\,{\cal C}_n\left((n-1)x\right),
\end{equation}

\begin{theorem}\label{Thm:Cx}
  The asymptotic distribution ${\cal C}(x)$ is given by
\begin{equation}\nonumber
  {\cal C}(x) = \frac{1}{\sqrt{x}}-1.
\end{equation}
\end{theorem}
\begin{proof}
The most satisfying proof of this fact is to show that the large-$n$
limits of the factorial moments are correctly reproduced. To wit,
\begin{equation}\nonumber
  \int_0^{1}dx\,x^m \left(\frac{1}{\sqrt{x}}-1\right)
 =\frac{1}{(m+1)(2m+1) }.
\end{equation}
Comparing to the result of Lemma \ref{Lem:Cfact}, we see that in the
large-$n$ limit $\frac{(n-1)!}{(n-m-1)!} \to n^m$, and so we are
indeed recovering the factorial moments correctly.
\end{proof}

\noindent{\bf Alternative proof} We use the same method presented in
the alternate proof of Theorem \ref{Thm:Kx}. Beginning with the exact distribution
\begin{equation}\nonumber
 \int_0^1 d\alpha \, 2(1-\alpha) \,{n-1\choose p}
\left(\alpha^2\right)^p \left( 1 -  \alpha^2\right)^{n-p-1},
\end{equation}
we approximate using an integral over Normal distributions with mean
$(n-1)\alpha^2$ and variance $(n-1)\alpha^2(1-\alpha^2)$
\begin{equation}\nonumber
  N(x) = \sqrt{\frac{n-1}{2\pi}} \int_0^1 d\alpha\,
  \frac{2(1-\alpha)}{\sqrt{\alpha^2(1-\alpha^2)}}
  \,\mathrm{Exp}\left( -\frac{(n-1)\left(x-\alpha^2\right)^2}
         {2\alpha^2(1-\alpha^2)}\right).
\end{equation}
There is a single saddle point at $\alpha=\alpha_0=\sqrt{x}$, and the
method of steepest descent proceeds as follows. Representing $N(x)$ as
\begin{equation}\nonumber
\int d\alpha\,f(\alpha) \,e^{(n-1) S(\alpha)},
\end{equation}
one finds that
\begin{equation}\nonumber
  \left. \frac{d^2S}{d\alpha^2} \right|_{\alpha = \alpha_0}=
  -\frac{4\alpha_0^2}{\alpha_0^2(1-\alpha_0^2)}.
\end{equation}
The contribution to the integral is then 
\begin{equation}\nonumber
  \frac{\sqrt{2\pi}}{\sqrt{n-1}}\,
  f(\alpha_0)\,
  \left(-\left. \frac{d^2S}{d\alpha^2} \right|_{\alpha = \alpha_0}\right)^{-1/2}
  =\frac{1-\alpha_0}{\alpha_0} = \frac{1}{\sqrt{x}}-1,
\end{equation}
where we have used the fact that $S(\alpha_0)=0$.
%%%%%%%%%%%%%%%%%%%%%%%%%%%%%%%%%%%%%%%%%%%%%%%%%%%%%%%%%%%%%%%%%%%%%%%%%
\section{Enumeration by containing pairs}

We remind the reader that a {\it containing} pair as a pair whose left
endpoint is left of the given pair's left endpoint, and whose right
endpoint is right of the given pair's right endpoint.

\begin{figure}[ht]
\begin{center}
 \includegraphics[bb=150 595 420 670, clip=true, width=4.0in]{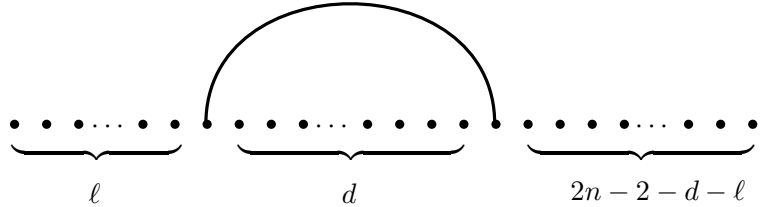}
\end{center}
\caption{Parameters used in the proof of Proposition \ref{Prop:Hnq};
  the given pair is indicated by the arc.}
\label{Fig:Hnq}
\end{figure}
\begin{proposition}\label{Prop:Hnq}
  The number $H_{n,q}$ of configurations with at least $q$ containing pairs is given by
  \begin{equation}\nonumber
    H_{n,q} = \sum_{d=0}^{2n-2-2q} \sum_{\ell = q}^{2n-d-2-q}
    {\ell \choose q}{2n-d-\ell-2\choose q}\,q!\,(2n-2q-3)!!.
    \end{equation}
\end{proposition}
\begin{proof}
We begin by parameterising the size and position of the given pair as
indicated in Figure \ref{Fig:Hnq}. We select $q$ vertices from the set
of $\ell$ vertices to the left of the given pair, and also a further
$q$ vertices from the set of $2n-2-d-\ell$ vertices to the right of
the given pair, then match them in all possible ways. The remaining
vertices are matched amongst themselves in all possible ways. We note that
\begin{itemize}
\item There are ${\ell \choose q}$ ways of selecting $q$ vertices from the $\ell$.
\item There are ${2n-d-\ell-2\choose q}$ ways of selecting $q$
  vertices from the  $2n-2-d-\ell$.
\item There are $q!$ ways of matching these two sets of $q$ vertices.
\item There are $(2n-2q-3)!!$ ways to match the remaining $2n-2q-2$ vertices.
\end{itemize}
These enumerations correspond to the factors of the summand; the sum
is over all possible values of position $\ell$ and size $d$ of the given pair.
\end{proof}
\begin{table}[t]
\begin{center}
  \begin{tabular}{c|llllll}
  $n$ \textbackslash$p$& 0& 1& 2& 3& 4& 5\\
  \hline  
  1& 1\\
  2& 5& 1\\
  3& 32& 11& 2\\
  4& 260& 116& 38& 6\\
  5& 2589& 1344& 594& 174& 24\\
  6& 30669& 17529& 9294& 3774& 984& 120\\
\end{tabular}
\end{center}
\caption{The numbers $G_{n,p}$, \seqnum{A336600} in the OEIS, to appear. The
  leading diagonal are the factorials $(n-1)!$. The first sub-leading diagonal
  is \seqnum{A001344}.}
\end{table}
\begin{lemma}\label{Lem:Gnp}
  The number $G_{n,p}$ of configurations with exactly $p$ containing pairs is given by
  \begin{equation}\nonumber
  G_{n,p}=n\,(2n-1)!!\,[y^p] \int_0^1 d\alpha\,
  \frac{\left(1+2\alpha(1-\alpha)(y-1)\right)^n-1}{2n\alpha(1-\alpha)(y-1)}.
  \end{equation}
\end{lemma}
\begin{proof}
We begin with the result of Proposition \ref{Prop:Hnq}, and shift the
summation variable, defining $m=\ell-q$, so that
\begin{equation}\nonumber
  H_{n,q} = \sum_{d=0}^{2n-2-2q} \sum_{m = 0}^{2n-2q-d-2} \frac{(m+q)!}{m!\,q!}
  \frac{(2n-d-m-q-2)!}{(2n-d-m-2q-2)!}\, (2n-2q-3)!!.
\end{equation}
We exploit the Euler Beta integral used in the proof of Lemma \ref{Lem:Knp} to obtain
\begin{equation}\nonumber
\begin{split}
  H_{n,q} &= \sum_{d=0}^{2n-2-2q} \sum_{m = 0}^{2n-2q-d-2}\int_0^1d\alpha\,
  \frac{\alpha^{m+q}}{m!\,q!}
  \frac{(1-\alpha)^{2n-d-m-q-2}\,(2n-d-1)!}{(2n-d-m-2q-2)!}\, (2n-2q-3)!!\\
& =
  \sum_{d=0}^{2n-2-2q}\int_0^1d\alpha\,
  \frac{\alpha^{q}(1-\alpha)^{2n-d-q-2}}{(2n-d-2q-2)!}
    \frac{(2n-d-1)!}{q!} \, (2n-2q-3)!!\\
    & \qquad \times\sum_{m =
      0}^{2n-2q-d-2}{2n-d-2q-2\choose
      m}\left(\frac{\alpha}{1-\alpha}\right)^m ,
\end{split}
  \end{equation}
where we have rearranged the summand to make the binomial nature of
the sum over $m$ manifest; performing this sum we obtain
\begin{equation}\nonumber
\begin{split}
  H_{n,q} &=
  \sum_{d=0}^{2n-2-2q}\int_0^1d\alpha\,
  \frac{\alpha^{q}(1-\alpha)^{2n-d-q-2}}{(2n-d-2q-2)!}
    \frac{(2n-d-1)!}{q!} \, (2n-2q-3)!!
    \left(1+\frac{\alpha}{1-\alpha}\right)^{2n-d-2q-2}\\
    &=  \frac{(2n-2q-3)!!}{q!}
    \int_0^1d\alpha\,
  \alpha^{q}(1-\alpha)^{q}
    \sum_{d=0}^{2n-2-2q}
    \frac{(2n-d-1)!}{(2n-d-2q-2)!}\\
    &= \frac{(2n-2q-3)!!}{q!}
    \int_0^1d\alpha\,
    \alpha^{q}(1-\alpha)^{q}\, \frac{n(2n-1)!}{(q+1)(2n-2q-2)!}\\
    &=n\,(2n-1)!!
    \int_0^1d\alpha\,\frac{1}{q+1}{n-1 \choose q}
\left(2\alpha(1-\alpha)\right)^q .
\end{split}
\end{equation}
We now form a generating function by summing over $q$ against $y^q$
\begin{equation}\nonumber
  H_n(y) = \sum_{q=0}^{n-1} H_{n,q}\,y^q =n\,(2n-1)!!
  \int_0^1 d\alpha\,
  \frac{\left(1+2\alpha(1-\alpha)y\right)^n-1}{2n\alpha(1-\alpha)y}.
\end{equation}
Finally we note that by inclusion-exclusion (c.f. \cite{Wilf}),
$G_{n,p}=[y^p]H_n(y-1)$, which yields the desired result.
\end{proof}
\begin{theorem}\label{Thm:Gyz}
The exponential generating function for the
numbers $G_{n,p}$ is given by
\begin{equation}\nonumber
  \sum_{n,p} G_{n,p}\frac{z^n}{n!} y^p =
  \frac{1}{(1-y)\sqrt{1-2z}}\ln\left(\frac{1-z(1+y)}{1-2z}\right).
\end{equation}
\end{theorem}
\begin{proof}
We sum the result of Lemma \ref{Lem:Gnp} against $z^n/n!$, and then
perform the integral over $\alpha$
\begin{equation}\nonumber
  \begin{split}
  &\sum_{n,p} G_{n,p}\frac{z^n}{n!} y^p
  =\int_0^1 d\alpha\,\sum_n \frac{(2n-1)!!}{n!}
  \frac{\left(1+2\alpha(1-\alpha)(y-1)\right)^n-1}{2\alpha(1-\alpha)(y-1)}\\
  &=\int_0^1 d\alpha\,\frac{1}{2\alpha(1-\alpha)(y-1)}\left(
  \frac{1}{\sqrt{1-2z\left(1+2\alpha(1-\alpha)(y-1)\right)}}
  -\frac{1}{\sqrt{1-2z}}\right).
%  &=\frac{1}{(1-y)\sqrt{1-2z}}\ln\left(\frac{1-z(1+y)}{1-2z}\right).
  \end{split}
\end{equation}
We use a Feynman parameter (c.f. \cite{S}) $\beta$ to combine the
denominator outside the parenthesis with those inside
\begin{equation}\nonumber
  \begin{split}
    \frac{1}{2(y-1)}\int_0^1d\beta\,\frac{1}{\sqrt{1-\beta}} \int_0^1 d\alpha\,
    &\Biggl( \frac{1}{\left(2\alpha(1-\alpha)\beta + (1-\beta)
      \left(1-2z\left(1+2\alpha(1-\alpha)(y-1)\right)\right)\right)^{3/2}}\\
    &- \frac{1}{\left(2\alpha(1-\alpha)\beta + (1-\beta)
      (1-2z)\right)^{3/2}} \Biggr).
 \end{split}
\end{equation}
The integral over $\alpha$ is straightforward and yields
\begin{equation}\nonumber
\begin{split}
  \frac{1}{(y-1)\sqrt{1-2z}}\int_0^1d\beta\,\Biggl(
  \frac{1}{(1-\beta)
    \left(2 - \beta-2(1-\beta)(1+y)z\right)}
  -\frac{1}{(1-\beta)
    \left(2 - \beta-4(1-\beta)z\right)}\Biggr),
  \end{split}
\end{equation}
where the apparent singularity at $\beta = 1$ cancels between the two
terms. The integration over $\beta$ is trivial and yields the desired
result.
\end{proof}

%---------------------------------------------------------------------------%
\subsection*{Probability distribution and asymptotics}

We define a discrete random variable $G$ which corresponds to the
number of pairs which are contained by the given pair. The result of Lemma
\ref{Lem:Gnp} implies that the probability that $G$ takes the value
$p$ is given by
\begin{equation}\nonumber
{\cal G}_{n,p}=[y^p] \int_0^1 d\alpha\,
  \frac{\left(1+2\alpha(1-\alpha)(y-1)\right)^n-1}{2n\alpha(1-\alpha)(y-1)}.
\end{equation}
In order to compute the factorial moments of this distribution, we
define a generating function as follows
\begin{equation}\nonumber
  \begin{split}
  {\cal P}_n(y)=\sum_{p=0}^{n-1} {\cal G}_n(p)\, y^p &= \int_0^1
  d\alpha\,
  \frac{\left(1+2\alpha(1-\alpha)(y-1)\right)^n-1}{2n\alpha(1-\alpha)(y-1)}\\
  &=\int_0^1 d\alpha\,\sum_{m=0}^{n-1}
  \left(2\alpha(1-\alpha)(y-1)\right)^m{n-1\choose m}\frac{1}{m+1}\\
  &=\sum_{m=0}^{n-1}
  \frac{2^m(m!)^2}{(2m+1)!}(y-1)^m{n-1\choose m}\frac{1}{m+1}\\
  &=\sum_{m=0}^{n-1}\frac{(n-1)!}{(n-m-1)!}\frac{2^mm!}{(m+1)(2m+1)! }(y-1)^m.
\end{split}
  \end{equation}
\begin{lemma}\label{Lem:Gfact}
The $m^\text{th}$ factorial moment of ${\cal G}_n(p)$ is given by
\begin{equation}\nonumber
\sum_{p=0}^{n-1} \frac{p!}{(p-m)!}\,{\cal G}_n(p)=
  \frac{(n-1)!}{(n-m-1)!}\frac{m!}{(m+1)(2m+1)!! }.
\end{equation}
In particular this provides the mean $E(G)=(n-1)/6$, and the variance
$\mathrm{Var}(G)=(n-1)(3n+19)/180$.
\end{lemma}
\begin{proof}
  Using the form of ${\cal P}_n(y)$ given above, we find
\begin{equation}\nonumber
\sum_{p=0}^{n-1} \frac{p!}{(p-m)!}\,{\cal G}_n(p)
= \left.\frac{d^m}{dy^m}\right|_{y=1} {\cal P}_n(y)
= \frac{(n-1)!}{(n-m-1)!}\frac{2^m(m!)^2}{(m+1)(2m+1)! },
\end{equation}
which yields the desired result upon simplification.
\end{proof}
In the limit as $n\to\infty$ we define a continuous real variable
$x=\lim_{n\to\infty}p/(n-1) \in [0,1]$, and an associated continuous
probability distribution
\begin{equation}\nonumber
  {\cal G}(x) = \lim_{n\to\infty} (n-1)\,{\cal G}_n\left((n-1)x\right).
\end{equation}
We note the similarity in the factorial moments between ${\cal
  G}_n(p)$ and ${\cal K}_n(p)$ (see Lemma \ref{Lem:Kfact}); indeed
those of ${\cal G}_n(p)$ are equal to $1/(m+1)$ times those of ${\cal
  K}_n(p)$. The following lemma allows us to exploit this fact to
determine the functional form of ${\cal G}(x)$.
\begin{lemma}\label{Lem:IBP}
  Let ${\cal P}(x)$ be a distribution with support on $x\in [a,b]$. Then  
\begin{equation}\nonumber
  \frac{1}{m+1}\int_a^b dx\,x^m{\cal P}(x) = \int_0^bdx\, x^m c_b
  -\int_0^adx\, x^m c_a-
  \int_a^bdx\,x^m \int^xdy\,\frac{{\cal P}(y)}{y},
\end{equation}
holds true, assuming the integrals are convergent. The constants
$c_a$ and $c_b$ are given by
\begin{equation}\nonumber
  c_a = \int^a dy\,\frac{{\cal P}(y)}{y}, \quad
  c_b = \int^b dy\,\frac{{\cal P}(y)}{y}.
\end{equation}
\end{lemma}
\begin{proof}
We begin with the last term on the right hand side and apply
integration by parts, integrating $x^m$ and differentiating $\int dy
\,y^{-1}{\cal P}(y)$
\begin{equation}\nonumber
  \int_a^bdx\,x^m \int^xdy\,\frac{{\cal P}(y)}{y} =
  \frac{b^{m+1}}{m+1} \int^b dy\,\frac{{\cal P}(y)}{y} -
  \frac{a^{m+1}}{m+1} \int^a dy\,\frac{{\cal P}(y)}{y} -
  \frac{1}{m+1}\int_a^b dx\,x^m{\cal P}(x).
\end{equation}
We then re-express the boundary terms as integrals over $x$, and
obtain the desired result.
\end{proof}

\begin{theorem}\label{Thm:Gx}
The asymptotic distribution ${\cal G}(x)$ is given by
\begin{equation}\nonumber
  {\cal G}(x)=
  \begin{cases}
    2\tanh^{-1}\sqrt{1-2x}& 0\leq x < 1/2 \\
    0& 1/2 \leq x \leq 1
    \end{cases}
\end{equation}
\end{theorem}
\begin{proof}
We use Lemma \ref{Lem:IBP}, letting ${\cal P}(x) = {\cal K}(x)$ from
Theorem \ref{Thm:Kx}, in order to deduce the distribution which
produces moments which are those of ${\cal K}(x)$ dressed by
$(m+1)^{-1}$. We note that
\begin{equation}\nonumber
  \int^x dy\, \frac{{\cal K}(y)}{y}
  = \int^x \frac{dy}{y\sqrt{1-2y}} =
  -2\tanh^{-1}\sqrt{1-2x}.
\end{equation}
The boundary terms are zero as $c_b=0$ since $b=1/2$ whilst $a=0$. 
\end{proof}

%%%%%%%%%%%%%%%%%%%%%%%%%%%%%%%%%%%%%%%%%%%%%%%%%%%%%%%%%%%%%%%%%%%%%%%%%
\section{Enumeration by excluded pairs}
We remind the reader that an {\it excluded} pair as a pair whose left
and right endpoints are both either to the left of the given pair's
left endpoint, or to the right of the given pair's right endpoint.

\begin{proposition}\label{Prop:Ynqr}
  The number $Y_{n,q,r}$ of configurations with at least $q$
  excluded pairs to the left of the given pair, and at least $r$
  excluded pairs to the right of the given pair, is given by
  \begin{equation}\nonumber
    Y_{n,q,r} = \sum_{d=0}^{2n-2-2q-2r} \sum_{\ell = 2q}^{2n-d-2-2r}
    {\ell \choose 2q}(2q-1)!!{2n-d-\ell-2\choose 2r}\,(2r-1)!!\,(2n-2q-2r-3)!!.
    \end{equation}
\end{proposition}
\begin{proof}
We begin by parameterising the size and position of the given pair as
indicated in Figure \ref{Fig:Hnq}. We select $2q$ vertices from the
set of $\ell$ vertices to the left of the given pair, and match them
amongst themselves in all possible ways. Similarly, we select $2r$
vertices from the set of $2n-2-d-\ell$ vertices to the right of the
given pair, and match them amongst themselves in all possible
ways. The remaining vertices are matched amongst themselves in all
possible ways. We note that
\begin{itemize}
\item There are ${\ell \choose 2q}$ ways of selecting $2q$ vertices from the $\ell$.
\item There are $(2q-1)!!$ ways of matching these vertices amongst themselves.
\item There are ${2n-d-\ell-2\choose 2r}$ ways of selecting $2r$
  vertices from the $2n-2-d-\ell$.
\item There are $(2r-1)!!$ ways of matching these vertices amongst themselves.
\item There are $(2n-2q-2r-3)!!$ ways to match the remaining $2n-2q-2r-2$ vertices.
\end{itemize}
These enumerations correspond to the factors of the summand; the sum
is over all possible values of position $\ell$ and size $d$ of the given pair.
\end{proof}
\begin{table}[t]
\begin{center}
  \begin{tabular}{c|llllll}
  $n$ \textbackslash$p$& 0& 1& 2& 3& 4& 5\\
  \hline  
  1& 1\\
  2& 4& 2\\
  3& 22& 16& 7\\
  4& 160& 136& 88& 36\\
  5& 1464& 1344& 1044& 624& 249\\
  6& 16224& 15504& 13344& 9624& 5484& 2190\\
\end{tabular}
\end{center}
\caption{The numbers $X_{n,p}$, \seqnum{A336601} in the OEIS, to appear. The
  first column is \seqnum{A087547}, the leading diagonal is \seqnum{A034430}.}
\end{table}
\begin{lemma}\label{Lem:Xnp}
  The number $X_{n,p}$ of configurations with exactly $p$ excluded pairs is given by
  \begin{equation}\nonumber
  X_{n,p}=n\,(2n-1)!!\,[y^p] \int_0^1 d\alpha\,
  \frac{\left(1+(1-2\alpha(1-\alpha))(y-1)\right)^n-1}{n(1-2\alpha(1-\alpha))(y-1)}.
  \end{equation}
\end{lemma}
\begin{proof}
We begin with the result of Proposition \ref{Prop:Ynqr}, and shift the
summation variable, defining $m=\ell-2q$, so that
\begin{equation}\nonumber
  \begin{split}
  Y_{n,q,r} &= \sum_{d=0}^{2n-2-2q-2r} \sum_{m = 0}^{2n-d-2-2r-2q}
  {m+2q \choose 2q}(2q-1)!!\\
  &\qquad\qquad\times{2n-d-m-2q-2\choose 2r}\,(2r-1)!!\,(2n-2q-2r-3)!!\\
  &= \frac{(2n-2q-2r-3)!!}{2^{q+r}\,q!\,r!}
  \sum_{d=0}^{2n-2-2q-2r} \sum_{m = 0}^{2n-d-2-2r-2q}
  \frac{(m+2q)!\,(2n-2-d-m-2q)!}
       {m! \,(2n-2-2q-2r-d-m)!}
  \end{split}
\end{equation}
We exploit the Euler Beta integral used in the proof of Lemma \ref{Lem:Knp} to obtain
\begin{equation}\nonumber
\begin{split}
  Y_{n,q,r} &=  \frac{(2n-2q-2r-3)!!}{2^{q+r}\,q!\,r!}
  \sum_{d=0}^{2n-2-2q-2r} 
(2n-d-1)!\\
  &\qquad\times\sum_{m = 0}^{2n-d-2-2r-2q}\int_0^1d\alpha\,\frac{\alpha^{m+2q}
  (1-\alpha)^{2n-d-m-2q-2}\,}{m!\,(2n-d-m-2q-2r-2)!}\\
& =\frac{(2n-2q-2r-3)!!}{2^{q+r}\,q!\,r!}\sum_{d=0}^{2n-2-2q-2r} 
\frac{(2n-d-1)!}{(2n-d-2q-2r-2)!} \,
  \int_0^1d\alpha\,
  \alpha^{2q}(1-\alpha)^{2n-d-2q-2}\\
    & \qquad \times\sum_{m = 0}^{2n-d-2-2r-2q}
         {2n-d-2q-2r-2\choose
      m}\left(\frac{\alpha}{1-\alpha}\right)^m ,
\end{split}
  \end{equation}
where we have rearranged the summand to make the binomial nature of
the sum over $m$ manifest; performing this sum we obtain
\begin{equation}\nonumber
\begin{split}
  Y_{n,q,r} &=
  \frac{(2n-2q-2r-3)!!}{2^{q+r}\,q!\,r!} \sum_{d=0}^{2n-2-2(q+r)}
 \frac{(2n-d-1)!}{(2n-d-2q-2r-2)!}\\
  & \qquad \times\int_0^1d\alpha\,
  \alpha^{2q}(1-\alpha)^{2n-d-2q-2}
   \, 
    \left(1+\frac{\alpha}{1-\alpha}\right)^{2n-d-2q-2r-2}\\
    &=   \frac{(2n-2q-2r-3)!!}{2^{q+r}\,q!\,r!} 
    \int_0^1d\alpha\,
  \alpha^{2q}(1-\alpha)^{2r}
    \sum_{d=0}^{2n-2-2q}
    \frac{(2n-d-1)!}{(2n-d-2q-2r-2)!}\\
    &=  \frac{(2n-2q-2r-3)!!}{2^{q+r}\,q!\,r!} 
    \int_0^1d\alpha\,
    \alpha^{2q}(1-\alpha)^{2r}\, \frac{n(2n-1)!}{(q+r+1)(2n-2q-2r-2)!}\\
    &=\frac{n\,(2n-1)!!}{q+r+1}{n-1 \choose q,r}
    \int_0^1d\alpha\,
    \alpha^{2q}(1-\alpha)^{2r}.
\end{split}
\end{equation}
We now form a generating function by summing both $q$ and $r$ against $y^{q+r}$
\begin{equation}\nonumber
  Y_n(y) = \sum_{q,r=0}^{n-1} Y_{n,q,r}\,y^{q+r} =n\,(2n-1)!!
  \int_0^1 d\alpha\,
  \frac{\left(1+(1-2\alpha(1-\alpha))y\right)^n-1}{n(1-2\alpha(1-\alpha))y}.
\end{equation}
Finally we note that by inclusion-exclusion (c.f. \cite{Wilf}),
$X_{n,p}=[y^p]Y_n(y-1)$, which yields the desired result.
\end{proof}
\begin{theorem}\label{Thm:Xyz}
The exponential generating function for the
numbers $X_{n,p}$ is given by
\begin{equation}\nonumber
X(y,z)=
  \sum_{n,p} X_{n,p}\frac{z^n}{n!} y^p =
  \frac{1}{(1-y)\sqrt{1-2z}}
  \tan^{-1}\frac{(1-y)z}{\sqrt{(1-2z)(1-2yz)}}.
\end{equation}
\end{theorem}
\begin{proof}
  We sum the result of Lemma \ref{Lem:Xnp} against $z^n/n!$, and then
perform the integral over $\alpha$
\begin{equation}\nonumber
  \begin{split}
  &\sum_{n,p} X_{n,p}\frac{z^n}{n!} y^p
  =\int_0^1 d\alpha\,\sum_n \frac{(2n-1)!!}{n!}
  \frac{\left(1+(1-2\alpha(1-\alpha))(y-1)\right)^n-1}{(1-2\alpha(1-\alpha))(y-1)}\\
  &=\int_0^1 d\alpha\,\frac{1}{(1-2\alpha(1-\alpha))(y-1)}\left(
  \frac{1}{\sqrt{1-2z\left(1+(1-2\alpha(1-\alpha))(y-1)\right)}}
  -\frac{1}{\sqrt{1-2z}}\right).
  \end{split}
\end{equation}
We change the integration variable to $x$, where $x^2=1-4\alpha(1-\alpha)$, yielding
\begin{equation}\nonumber
  \begin{split}
  X(y,z) &= \frac{1}{(1-y)^{3/2}\sqrt{z}}\int_{-1}^1dx\,\frac{1}{1+x^2}\left(
  \frac{1}{\sqrt{A-1}}-\frac{1}{\sqrt{A+x^2}}  \right),\\
  &=\frac{1}{(1-y)^{3/2}\sqrt{z}}\left(\frac{\pi}{2\sqrt{A-1}}
  -\int_{-1}^1dx\,\frac{1}{1+x^2}
  \frac{1}{\sqrt{A+x^2}}\right).
  \end{split}
\end{equation}
where $A = (1-(1+y)z)/(z(1-y))$. A final change of variable to $u$,
where $\tan u = x\,\sqrt{A-1}/\sqrt{A+x^2}$ renders the remaining
integral trivial
\begin{equation}\nonumber
\begin{split}
  \int_{-1}^1dx\,\frac{1}{1+x^2}
\frac{1}{\sqrt{A+x^2}} &= \frac{1}{\sqrt{A-1}}
\int_{-\tan^{-1}\sqrt{(A-1)/(A+1)}}^{\tan^{-1}\sqrt{(A-1)/(A+1)}}du
=\frac{2}{\sqrt{A-1}} \tan^{-1} \sqrt{\frac{A-1}{A+1}}\\
&=\frac{1}{\sqrt{A-1}} \tan^{-1} \sqrt{A^2-1},
\end{split}
\end{equation}
where in the last equality we have exploited the double angle formula
for $\tan$. We thus obtain
\begin{equation}\nonumber
  \begin{split}
  X(y,z)   &=\frac{1}{(1-y)^{3/2}\sqrt{z}}\left(\frac{\pi}{2\sqrt{A-1}}
    -\frac{1}{\sqrt{A-1}} \tan^{-1} \sqrt{A^2-1}
    \right)\\
    &=\frac{1}{(1-y)^{3/2}\sqrt{z}}\frac{1}{\sqrt{A-1}}
    \tan^{-1} \frac{1}{\sqrt{A^2-1}},
  \end{split}
\end{equation}
which yields the desired result.
\end{proof}

%---------------------------------------------------------------------------%
\subsection*{Probability distribution and asymptotics}

We define a discrete random variable $X$ which corresponds to the
number of pairs which are excluded by the given pair. The result of Lemma
\ref{Lem:Xnp} implies that the probability that $X$ takes the value
$p$ is given by
\begin{equation}\nonumber
{\cal X}_{n,p}=[y^p] \int_0^1 d\alpha\,
  \frac{\left(1+(1-2\alpha(1-\alpha))(y-1)\right)^n-1}{n(1-2\alpha(1-\alpha))(y-1)}.
\end{equation}
In order to compute the factorial moments of this distribution, we
define a generating function as follows
\begin{equation}\nonumber
  \begin{split}
  {\cal P}_n(y)=\sum_{p=0}^{n-1} {\cal X}_n(p)\, y^p &= \int_0^1
  d\alpha\,
  \frac{\left(1+(1-2\alpha(1-\alpha))(y-1)\right)^n-1}{n(1-2\alpha(1-\alpha))(y-1)}\\
  &=\int_0^1 d\alpha\,\sum_{m=0}^{n-1}
  \left(1-2\alpha(1-\alpha)(y-1)\right)^m{n-1\choose m}\frac{1}{m+1}\\
  &=\sum_{m=0}^{n-1}\frac{(n-1)!}{(n-m-1)!}\frac{1}{m+1}\frac{(y-1)^m}{m!}
  \int_0^1 d\alpha\,\left(1+(1-2\alpha(1-\alpha))\right)^m.
  \end{split}
\end{equation}
\begin{lemma}\label{Lem:Xfact}
The $m^\text{th}$ factorial moment of ${\cal X}_n(p)$ is given by
\begin{equation}\nonumber
\sum_{p=0}^{n-1} \frac{p!}{(p-m)!}\,{\cal X}_n(p)=
  \frac{(n-1)!}{(n-m-1)!}\frac{1}{m+1}\int_{1/2}^1dx\,\frac{x^m}{\sqrt{2x-1}}.
\end{equation}
In particular this provides the mean $E(X)=(n-1)/3$, and the variance
$\mathrm{Var}(X)=2(n-1)(n+3)/45$.
\end{lemma}
\begin{proof}
We use the form of ${\cal P}_n(y)$ given above to obtain
\begin{equation}\nonumber
\sum_{p=0}^{n-1} \frac{p!}{(p-m)!}\,{\cal X}_n(p)
= \left.\frac{d^m}{dy^m}\right|_{y=1} {\cal P}_n(y)
=\frac{(n-1)!}{(n-m-1)!}\frac{1}{m+1}
  \int_0^1 d\alpha\,\left(1+(1-2\alpha(1-\alpha))\right)^m. 
\end{equation}
We change the integration variable to $x=1-2\alpha(1-\alpha)$, and
obtain the desired result.
\end{proof}

In the limit as $n\to\infty$ we define a continuous real variable
$x=\lim_{n\to\infty}p/(n-1) \in [0,1]$, and an associated continuous
probability distribution
\begin{equation}\nonumber
  {\cal X}(x) = \lim_{n\to\infty} (n-1)\,{\cal X}_n\left((n-1)x\right).
\end{equation}

\begin{theorem}\label{Thm:Xx}
The asymptotic distribution ${\cal X}(x)$ is given by
\begin{equation}\nonumber
  {\cal X}(x)=
  \begin{cases}
    \pi/2& 0\leq x < 1/2 \\
    \pi/2-2\tan^{-1}\sqrt{2x-1}& 1/2 \leq x \leq 1
    \end{cases}
\end{equation}
\end{theorem}
\begin{proof}
We use Lemma \ref{Lem:IBP}, letting ${\cal P}(x) = (2x-1)^{-1/2}$ from
the integrand of Lemma \ref{Lem:Xfact}, in order to deduce the
distribution which produces moments which are those of $(2x-1)^{-1/2}$
dressed by $(m+1)^{-1}$. We note that
\begin{equation}\nonumber
  \int^x dy\, \frac{1}{y\sqrt{2y-1}}=
  2\tan^{-1}\sqrt{2x-1}.
\end{equation}
We further note that $a=1/2$ and $b=1$, yielding $c_a=0$ and $c_b =
\pi/2$. Thus the distribution receives a constant contribution of
$\pi/2$ across the entire interval $x\in[0,1]$. By Lemma \ref{Lem:IBP}
we obtain the desired result.
\end{proof}

\end{document}